\documentclass[english]{amsart}
\usepackage{amsmath}
\usepackage{amsthm}
\usepackage{amssymb}
\usepackage[cal=euler]{mathalfa}
\usepackage{tikz-cd}
\tikzcdset{
arrow style=tikz,
diagrams={>={Computer Modern Rightarrow}}
}
\makeatletter
\usepackage{xcolor}
\PassOptionsToPackage{dvipsnames}{xcolor}
    \RequirePackage{xcolor}
\definecolor{Maroon}{cmyk}{0, 0.87, 0.68, 0.32}
\definecolor{RoyalBlue}{cmyk}{1, 0.50, 0, 0}
\usepackage[pdftex,breaklinks,colorlinks,
linkcolor=RoyalBlue,
citecolor=RoyalBlue,
urlcolor=black]{hyperref}
\theoremstyle{plain}
\newtheorem{thm}{\protect\theoremname}[section]
  \theoremstyle{definition}
  \newtheorem{defn}[thm]{\protect\definitionname}
  \theoremstyle{definition}
  \newtheorem{example}[thm]{\protect\examplename}
  \theoremstyle{plain}
  \newtheorem{prop}[thm]{\protect\propositionname}
  \newtheorem{lemma}[thm]{\protect\lemmaname}
  \theoremstyle{definition}
  
  \theoremstyle{plain}
  \newtheorem{cor}[thm]{\protect\corollaryname}
  \theoremstyle{remark}
  \newtheorem{rem}[thm]{\protect\remarkname}
  \newtheorem{paragr}[thm]{}

\@ifundefined{date}{}{\date{}}
\makeatother

\usepackage{babel}
  \providecommand{\corollaryname}{Corollary}
  \providecommand{\definitionname}{Definition}
  \providecommand{\examplename}{Example}
  \providecommand{\exercisename}{Exercise}
  \providecommand{\propositionname}{Proposition}
  \providecommand{\remarkname}{Remark}
  \providecommand{\theoremname}{Theorem}
  \providecommand{\lemmaname}{Lemma}

\global\long\def\lim{\operatorname*{lim}}
\global\long\def\map{\operatorname{Map}}

\newcommand{\ZZ}{\mathbf{Z}}
\newcommand{\QQ}{\mathbf{Q}}
\newcommand{\FF}{\mathbf{F}}
\newcommand{\PP}{\mathbf{P}}
\newcommand{\HH}{\mathbf{H}}

\newcommand{\et}{\mathrm{\acute{e}t}}

\newcommand{\uHom}{\mathit{Hom}}
\newcommand{\RHom}{\mathrm{RHom}}

\newcommand{\spec}[1]{\mathrm{Spec}(#1)}

\newcommand{\DM}{\mathit{DM}}
\newcommand{\h}{h}
\newcommand{\DMh}{\DM_\h}
\newcommand{\DMhlc}{\DM_{\h,lc}}

\newcommand{\unit}{\mathbf{1}}
\newcommand{\ind}{\mathrm{Ind}}
\newcommand{\huniv}{h_{\mathit{univ}}}

\title{Uniform approximation of Betti numbers}
\author[D.-C. Cisinski]{Denis-Charles Cisinski}
\address{Fakult{\"a}t f{\"u}r Mathematik,
	Universit{\"a}t Regensburg,
	93040 Regensburg,
	Germany}
\email{denis-charles.cisinski@ur.de}
\urladdr{https://cisinski.app.uni-regensburg.de/}

\begin{document}
\begin{abstract}
We prove that Lefschetz's principle of approximating the cohomology of a
possibly singular affine
scheme of finite type over a field by the cohomology of a suitable
(thickening of a)
hyperplane section can be made uniform: in the affine case,
we can choose the hyperplane section
independently of the cohomology. Using Jouanolou's trick,
this gives a new way to bound the Betti numbers of quasi-projective schemes
over a field, independently of the cohomology.
This is achieved through
a motivic version of Deligne's generic base change formula
and an axiomatic presentation of the theory of perverse sheaves.
These methods produce generating families of Voevodsky motivic sheaves
that are realized in perverse sheaves over
any base of equal characteristic.
\end{abstract}
%
%
\maketitle
\tableofcontents
\section*{Introduction}
Since Grothendieck introduced \'etale cohomology,
proving that the Betti numbers of $\ell$-adic cohomology
of algebraic varieties over a separably closed field
are independent of $\ell$
remains a deep open problem in positive characteristic.
The case of smooth and proper algeraic varieties
is reached by weight techniques, using Deligne's proof of the Weil conjectures,
but there is not any other general result in this direction.
The next best thing consists in providing uniform bounds of Betti
numbers. This can be done by d\'evissage from the case of
smooth and proper algebraic varieties,
and then applying descent techniques together with de Jong's alteration theorem,
as explained by Katz~\cite[Theorem~5]{Katz2}
who credits the result to de Jong and Berthelot.
There are results for non-constant coefficients as well, working with
compatible families \cite{KatzLaumon,Katz2,Orgo}.

The present notes give new proofs of existence of such uniform bound, using
motivic techniques (in particular,
we do not only prove independence of $\ell$ results
but rather independence against a large class of Weil cohomologies).
We apply Lefschetz's principle of approximating
a quasi-projective scheme $X$ by a suitable hyperplane section.
Of course, this needs adjustments in order to take into account the
singularities of $X$ as well as the fact that we want the choice of
the hyperplane section to be independent of the cohomology we work
with (e.g. $\ell$-adic cohomology of schemes of finite type over a
separably closed field $k$, with $\ell$ distinct from the characteristic of $k$,
or Berthelot's rigid cohomology if $k$ is perfect of characteristic $p>0$).

To be more precise, the general strategy goes as follows.
In this introduction, in order to simplify things,
we work over an infinite field $k$ of characteristic $p$.
Using Jouanolou's trick \cite[Lemme 1.5]{Jou}, since the formation
of Voevodsky motives is $\mathbf{A}^1$-homotopy invariant, the motive of any quasi-projective scheme is isomorphic to the motive of an affine one.
Therefore, we can focus on affine schemes of finite type over $k$.
For such a scheme $X$ of dimension $d$,
one defines the notion of Leschetz subscheme:
a closed subscheme $Z$ of $X$ of dimension $d-1$, with affine
complement, that contains all
irreducible components of $X$ of dimension $<d$ and the singular locus
of $X$, as well as
a hyperplane section satisfying a suitable transversality condition,
expressed in the language of the six operations for
motivic sheaves; see Definition \ref{def:Lefschetz filtration} below.
We define the motive of the pair $(X,Z)$ through the cofiber
sequence
\[
M(Z)\to M(X)\to M(X,Z)
\]
in Voevodsky's category $\DM_c(\spec k)$ (of
constructible \'etale motives).
We prove that, for any prime number $\ell\neq p$,
the $\ell$-adic realization of $M(X,Z)$ is concentrated in degree $-d$
and that its unique non-trivial cohomology group is a free $\ZZ_\ell$-module
of finite type, using the theory of $\ell$-adic
perverse sheaves\footnote{This is where the fact that we work with affine schemes matters: the main tool is then Artin's vanishing theorem in order
to prove Beilinson's transversality Theorem \ref{thm:Beilinson} below.};
see Lemma \ref{lemma:Lefschetz pairs are concentrated}.
Since the $\ell$-adic realization of motivic sheaves
is symmetric monoidal, it is compatible with the formation and the
computation of traces of endomorphisms of constructible motives,
in particular of Euler
characteristics (these are traces of identities).
In this case, this means that the Euler characteristic of $M(X,Z)[-d]$,
seen as a motive, coincides with the rank $\mathit{rk}(X,Z)$
of the unique non-trivial
cohomology group of its $\ell$-adic realization,
proving that the latter is independent of $\ell$.
This yields isomorphisms
\[
H^q_\et(\bar X,\ZZ_\ell)\cong H^q_\et(\bar Z,\ZZ_\ell)\quad\text{for $q\neq d, d-1$,}
\]
together with an exact sequence
\[
0\to H^{d-1}_\et(\bar X,\ZZ_\ell)
\to H^{d-1}_\et(\bar Z,\ZZ_\ell)\to \ZZ_\ell^{\mathit{rk}(X,Z)}\to
H^{d}_\et(\bar X,\ZZ_\ell)\to 0\, .
\]
In particular we get a bound of
the $d$th Betti number of $X$ by the motivically defined
non-negative integer $\mathit{rk}(X,Z)$,
and we observe that the Betti numbers of $X$ in lower degree are bounded
by those of $Z$.
This is meaningful because, following ideas from Beilinson~\cite{Beilinson},
we apply a motivic variant of Deligne's generic base change theorem,
to deduce a motivic shadow of the
weak Lefschetz theorem (Theorem \ref{thm:weak Lefschetz} below).
This implies that,
for any closed subscheme $Y\subset X$ of
dimension $<d$, there is a Lefchetz subscheme $Z\subset X$ containing $Y$
(this is where we need $k$ to be infinite: in fact, we prove that the family of
eligible hyperplane sections contains the rational points of an open
dense subscheme of the space of all hyperplane sections);
see Prop.~\ref{prop:existence Lefschetz pairs}.
Therefore, following a longstanding tradition in Algebraic Topology \cite{Milnor,Thom},
we can produce a Lefschetz filtration of $X$:
a finite sequence of closed subschemes
of the form
\[
\varnothing=X_{-1}\subset X_0\subset\ldots\subset X_{d-1}\subset X_d=X
\]
where, for $0\leq i\leq d$, each $X_{i-1}$ is a Lefschetz subscheme of $X_i$.
This leads to a filtration of the motive of $X$ whose associated graded
object yields a complex with terms of the form $M(X_i,X_{i-1})$ computing
the Euler characteristic of $X$
\[
\chi(X)=\sum_i(-1)^i\mathit{rk}(X_i,X_{i-1})
\]
and such that
the $i$th $\ell$-adic Betti number of the affine scheme $X$
(both $\QQ_\ell$- or with $\FF_\ell$-coefficients)
is bounded by $\mathit{rk}(X_i,X_{i-1})$ for all $\ell\neq p$;
see Theorem \ref{thm:Betti estimation}.

However, our proof does more than providing a bound that is
independent of $\ell$: it provides a uniform bound that is valid
in any realization of Voevodsky's motivic sheaves that has
a theory of perverse sheaves (including the properties that Verdier
duality induces an equivalence on constructible objects and that
Artin's vanishing theorem holds); see Definition \ref{def:Artin-Verdier}.
By developing this axiomatic approach, we also try to reach
a wide enough level of generality, geometrically:
instead of working over a field,
we develop enough material to prove all the main ingredients
over a scheme of finite type over a field: what is true over an infinite
field is in fact true generically, possibly after pulling
back along a flat finite cover of a
dense open subset of any algebraic variety (from this point
of view, we can work over any field $k$).
Because what is at stake, after all, is the problem of finding
big enough
families of motivic sheaves that are known to be realized
into ordinary sheaves or into
perverse sheaves.
This is what we do
at the end of this article (Prop.~\ref{prop:Lefschetz enough}
and \ref{prop:existence tight Lefschetz pairs}),
thus providing constructions of
$t$-structures on $\DM(X,\QQ)$ that are candidates for the
perverse motivic $t$-structure~\eqref{paragr:t-structures}.
These constructions, by finding nice generators for categories
of motivic sheaves, may be seen as a way
to provide uniform bounds of Betti numbers of all sorts,
not only for the cohomology
of algebraic varieties, but for the cohomology of any constructible
motivic sheaf (this includes
cohomology with compact support or not, as well
as nearby cycles, thanks to Ayoub's work \cite{ay}), possibly in relative situations.
Finding explicit generators for such $t$-structures is also part
of the standard tools to develop the theory of
Nori motives and to relate them with Voevodsky motives in characteristic zero.
Therefore, the results of the present text open the way to a
uniform theory of Nori motives \cite{HMS,Ara1,MI},
possibly over any field:
instead of working with a single cohomology theory, we can now consider the family
of all cohomology theories appearing in a six-functors formalism in which there is
a good notion of perverse sheaves. This is strongly related with Ayoub's
constructions~\cite{AyoubHopf3} of new Weil cohomologies in positive characteristic 
and their associated Hopf algebroids.
\nocite{Ara2}

\section{Motivic shadow of the weak Lefschetz Theorem}
\begin{paragr}\label{def:6FF}
We consider a $\ZZ$-linear six-functors formalism $D$ in the sense of
Voevodsky-Ayoub (see \cite{Univ6FF}), taking values in presentable $\infty$-categories, defined on the category of schemes locally of finite presentation over a base scheme $S$. 
We assume furthermore that, for any $X$, there is a full stable subcategory $D_c(X)$ of $D(X)$, whose objects are called \emph{constructible objects} with the following properties:
\begin{itemize}
\item[(i)] $D_c(X)$ contains the $\otimes$-unit $\unit_X$,
is closed under tensor product, under pullbacks of the form $f^*$ for any map $f$, and under direct image with compact support $f_!$ for any separated morphism of finite presentation $f$;
\item[(ii)] for any $X$ and any $F$ in $D_c(X)$, there is a dense open subscheme $U\subset X$ such that $F_{|U}$ is dualizable (in the sense of
\cite[ Section 4.6.1]{HA});
\item[(iii)] for any $X$, any dualizable object of $D(X)$ is constructible;
\item[(iv)] constructible objects form an hypersheaf of $\infty$-categories for the
\'etale topology.
\end{itemize}
Remark that, in the case of schemes with noetherian underlying space, these properties determine $D_c(X)$ uniquely.
\end{paragr}
\begin{example}\label{ex:def DM}
Let $\DM(X)=\ind(\DMhlc(X))$ the $\infty$-category of ind-objects on the $\infty$-category of locally constructible objects of the $\infty$-category of $\h$-motives
in the sense of \cite[Def. 6.3.1]{CD4}. Then $\DM$ satisfies all the properties above,
at least if we restrict to noetherian schemes (but this is not even necessary); this follows easily from see \cite[Theorem 6.3.26]{CD4}. 
For any commutative ring $\Lambda$, there is a $\Lambda$-linear version $\DM(X,\Lambda)$
defined as the $\Lambda$-modules in $\DM(X)$.

If $\Lambda$ is a commutative ring of characteristic $n>0$,
there is a rigidity theorem in the following sense:
for $n$ prime to the residue
characteristics of $X$, there is a canonical equivalence between compact $\Lambda$-modules in $\DM(X)$ and ordinary constructible sheaves of $\Lambda$-modules of finite tor-dimension over the small \'etale site of $X$:
\[
\DM_c(X,\Lambda)\cong D^b_{\mathrm{ctf}}(X,\Lambda)
\]
(see \cite[Theorem 6.3.11]{CD4}).
\end{example}
\begin{example}\label{ex:ladic}
Let $\ell$ be a prime number and define
\[
D^b_c(X,\ZZ_\ell)=\varprojlim_i \DMhlc(X,\ZZ/\ell^i\ZZ)
\]
where the limit is taken in the $\infty$-category of (stable) $\infty$-categories.
In the case where $\ell$ is prime to the residue characteristics of $X$,
the rigidity theorem imples that $D^b_c(X,\ZZ_\ell)$ is the usual derived category of
constuctible $\ell$-adic sheaves on the small \'etale site of $X$; see \cite[Prop.~7.2.21]{CD4}. The assignment
\[
X\mapsto D(X,\ZZ_\ell)=\ind(D^b_c(X,\ZZ_\ell))
\]
is then a six-functors formalism that satisfies all properties
listed in paragraph \ref{def:6FF}.
We observe that the dualizable objects
are then exactly those complexes whose cohomology objects are
what Deligne calls ``faisceaux localement constants tordus''.
Indeed, the functor
\[
D^b_c(X,\ZZ_\ell)\to D^b_c(X,\FF_\ell)\ , \quad F\mapsto F/\ell
\]
is conservative (by definition of what ``$\ell$-adic'' means)
and symmetric monoidal. This implies that it preserves and detects
dualizable objects. Therefore, by virtue of \cite[Remark 6.3.27]{CD4},
$F$ in $D^b_c(X,\ZZ_\ell)$ is dualizable if and only if $F/\ell$
is locally free with perfect values; for a discussion
using condensed techniques, see \cite{BS,HRS}.
We denote by $D(X,\QQ_\ell)$ the full subcategory of uniquely divisible
objects in $D(X,\ZZ_\ell)$. It is compactly generated with category of
compact objects 
\[
D^b(X,\QQ_\ell)=D^b_c(X,\ZZ_\ell)\otimes\QQ\, .
\]
\end{example}
\begin{example}\label{ex:ultra}
Let $\mathcal{U}$ be an ultrafilter on the set of prime numbers.
We define for any noetherian $X$:
\[
D^b_c(X,\FF_\mathcal{U})'=\varinjlim_{A\in\mathcal{U}}\prod_{\ell\notin A}\DMhlc(X,\FF_{\ell})
\]
(where the colimit and the product are taken in the $\infty$-category
of (stable) $\infty$-categories). 
We define $D^b_c(X,\FF_\mathcal{U})$ as the full subcategory
of $D^b_c(X,\FF_\mathcal{U})'$ by noetherian induction:
an object $F$ of $D^b_c(X,\FF_\mathcal{U})'$ belongs to
$D^b_c(X,\FF_\mathcal{U})$ if there is a dense open subscheme $V$ of $U$
such that the restriction $F_{|_V}$ is dualizable in $D^b_c(V,\FF_\mathcal{U})'$
and $F$ is constructible on a complement of $V$.
We define finally
\[
D(X,\FF_\mathcal{U})=\ind(D^b_c(X,\FF_\mathcal{U}))\, .
\]
If the set of residue characteristics of $X$ is finite, $D^b_c(X,\FF_\mathcal{U})'$
coincides with the $\mathcal U$-indexed ultraproduct of the ordinary categories of
constructible sheaves $D^b_c(X,\FF_\ell)$. 
If the set of residue characteristics is finite (e.g. is X is of finite type over a
semi-local scheme), then $D^b_c(X,\FF_\mathcal{U})'$ coincides with the
ultraproduct of the usual categories of \'etale sheaves $D(X,\FF_\ell)$,
where $\ell$ runs over all primes distinct from the residue characteristics of $X$.
The corresponding cohomology is then the ultraproduct in the usual sense of
\'etale cohomology groups:
\[
\pi_0\map_{D(X,\FF_\mathcal{U})}(\mathbf{1}_X,\mathbf{1}_X[i])\cong
\varinjlim_{A\in\mathcal{U}}\prod_{\ell\notin A}H^i_\et(X,\FF_\ell)\, .
\]
For an arbitrary scheme $X$, by virtue of \cite[Cor. 4.5.3]{CD4}, we always have
\[
\DMhlc(X,\FF_\ell)\cong D^b_c(X[1/\ell],\FF_\ell)
\]
where $X[1/\ell]$ is the product of $X$ with the spectrum of $\ZZ[\frac{1}{\ell}]$. Therefore, we have
\[
\pi_0\map_{D(X,\FF_\mathcal{U})}(\mathbf{1}_X,\mathbf{1}_X[i])\cong
\varinjlim_{A\in\mathcal{U}}\prod_{\ell\notin A}H^i_\et(X[1/\ell],\FF_\ell)\, .
\]
The only non-trivial property of \ref{def:6FF} that is difficult to check is (i). In the case of quasi-excellent schemes over a semi-local noetherian ring, this follows right away from Orgogozo's uniform constructibility theorem \cite[Th\'eor\`eme 3.1.1]{Orgo}.
\end{example}
\begin{example}
If the base field has a complex embedding $k\hookrightarrow\mathbf{C}$,
one can define the Betti six-functors formalism
\[
X\mapsto D_{\textit{Betti}}(X)=D(X(\mathbf{C}),\ZZ)
\]
by taking derived $\infty$-categories of sheaves of abelian groups
on the analytic variety associated to $X$. The associated
cohomology is Betti cohomology.
\end{example}
\begin{example}
If $k$ is of characteristic zero, one can define the de Rham
six-functors formalism over schemes of finite type over $k$
through the theory of algebraic $D$-modules.
The corresponding cohomology is algebraic de Rham cohomology.
\end{example}
\begin{example}
If $k$ is a perfect field of characteristic $p>0$, choosing a
complete discrete valuation ring of characteristic zero with residue field $k$
(such as the ring of Witt vectors), Berthelot's theory of arithmetic $D$-modules
defines a six-functors formalism for schemes of finite type over $k$.
The corresponding cohomology is then Berthelot's rigid cohomology
(i.e. it coincides with Monsky-Washnitzer cohomology for affine smooth $k$-schemes
and with crystalline cohomology for smooth and projective ones).
\end{example}
\begin{paragr}
If $D$ and $D'$ are two six-functors formalisms defined for schemes over $S$,
there is a natural notion of morphism of six-functors formalisms from $D$ to $D'$;
see \cite{Univ6FF}.
Informally, this consists of colimit-preserving (hence exact) symmetric monoidal functors
\[
D(X)\to D'(X)
\]
that are compatible with pullbacks $f^*$ for all $f$ as well with direct image with compact support $f_!$ for any $f$ separated and of finite presentation.
In particular, on noetherian schemes, such morphisms must preserve constructible
objets (they preserve dualizable objects, which implies that they preserve constructible objects by noetherian induction). One can prove that $\DM$ is the initial object among
$\ZZ$-linear
six-functors formalisms satisfying \'etale descent\footnote{One can see that $\DMh$ is initial among $\ZZ$-linear six-functors formalisms satisfying \'etale hyperdescent:
since we know what is the initial six-functors formalism by virtue of Drew and Gallauer's
\cite[Theorem 7.14]{Univ6FF},
we see easily that it suffices to prove it for $\QQ$-linear coefficients, which follows right away from \cite[Theorem 16.1.2 and 16.2.18]{CD3},
and for finite coefficients, which follows from the combination of
rigidity theorems: the one recalled in Example \ref{ex:def DM}
together with those stated in \cite{AyoubEtale,Bachmann}. This implies
that $\DM$ is the initial object among $\ZZ$-linear six-functors formalisms satisfying the properties of \ref{def:6FF}. For $X$ noetherian with residue fields that have a uniform bound on their Galois cohomological dimension, $\DM$ and $\DMh$ agree anyway.}. In particular, for any other
six-functor formalism $D$, there is a unique morphism
\[
r_D\colon\DM(-)\to D(-)\, .
\]
We call an object $F$ of $D(X)$ \emph{strictly of geometric origin} if there is a constructible
motivic sheaf $M$ in $\DM_c(X)$ such that $F\cong r_D(M)$ in $D(X)$ (in particular,
such an $F$ is constructible). We will say that $D$ \emph{behaves well} over $S$ if
the functor $r_D$ commutes with functors of the form $f_*$ for any morphism
of finite presentation $f$.
All examples of six-functors formalisms given in the present notes are known to
behave well: in the case where $D$ is $\QQ$-linear and defined on schemes
of finite type over a field,
this is subsumed by
\cite[Theorem 4.4.25]{CD3} -- see Remark \ref{rem:gbc} below
to drop the $\QQ$-linearity. For $D=D(-,\Lambda)$ with $\Lambda$ of positive
characteristic, this is obvious, by rigidity.
This is also true in the case where $D(X)=D(X,\ZZ_\ell)$
for any $S$ (see \cite[Theorem 7.2.11]{CD4}), or $D(X)=D(X,\FF_\mathcal{U})$ (exercise).

The property of being dualizable is at the core of many constructions and computations.
In particular, the fact that any construcible sheaf becomes dualizable
over a dense open subscheme is at the center of everything.
For objects that are strictly of geometric origin, the choice
of such a dualizibility locus can be made motivically,
hence, independently of $D$. This is made precise in the following statement.
\end{paragr}
\begin{prop}\label{prop:generic dualizability}
Let $D$ be six-functors formalism defined over $S$-schemes.
If $F$ is strictly of geometric origin in $D(X)$, then one can find
a dense open subscheme $U$ of $X$, independently of $D$,
such that the restriction $F_{|U}$ is dualizable in $D(U)$.
\end{prop}
\begin{proof}
Let $M$ be constructible in $\DM(X)$ such that $F\cong r_D(M)$.
We choose a dense open immersion $j\colon U\to X$ such that $M_{|U}=j^*(M)$
is dualizable in $\DM(X)$. Since the functor
\[
r_D\colon\DM(U)\to D(U)
\]
is symmetric monoidal, it preserves dualizable objects, and therefore
$F_{|U}\cong r_D(M_{|U})$ is dualizable. The open subscheme $U$ clearly
depends on $M$ only.
\end{proof}

\begin{paragr}
The following statement,
is a generalization of Deligne's generic base change theorem
for torsion \'etale sheaves \cite[Th.~Finitude, 1.9]{SGA4demi}.
With this new degree of generality, it may be seen as an independence
of $D$ (in particular, independence of $\ell$) result as well;
see Corollary \ref{cor:gbc}.
\end{paragr}
\begin{thm}[Deligne's generic base change formula]\label{thm:gbc}
Let $f:X\to Y$ be a morphism
between separated schemes of finite type over a noetherian base scheme $S$.
Let $F$ be a constructible object in $D(X)$. Then there
is a dense subscheme $U\subset S$ over which $f_{*}(F)$
becomes constructible and its formation
is compatible with any base-change.
More precisely, for each $w:S'\to S$ factoring through $U$ we have 
an isomorphism of constructible objects
\[
v^{*}f_{*}(F)\cong f'_{*}u^{*}(F)
\]
where 
\[
\begin{tikzcd}
X'\ar{r}{u}\ar{d}{f'} & X\ar{d}{f}\\
Y'\ar{r}{v}\ar{d} & Y\ar{d}\\
S'\ar{r}{w} & S
\end{tikzcd}
\]

is the associated pull-back diagram. 
\end{thm}
\begin{proof}
See \cite[Theorem 2.4.2]{MotGenBC} (the proof is stated for $D(X)=\DM(X)$
but it only uses the properties listed in paragraph \ref{def:6FF}).
The statement in \emph{loc. cit.} does not say that $f_*(F)$ becomes
constructible over $U$, but this is what the proof yields anyway.
\end{proof}
\begin{rem}
The proof of the generic base change formula really uses \'etale descent
on constructible sheaves:
this implies that, for any quasi-finite morphism $f$. the functors $f_!$ and $f_*$
are conservative, which is a crucial ingredient in Deligne's proof.
It is not known whether the generic base change formula is true
or not for more universal theories of motivic sheaves in which we only have
Nisnevich descent, such as Voevodsky's and Morel's $\mathit{SH}$.
\end{rem}
\begin{rem}\label{rem:gbc}
In fact, a close examination of the proof
of the generic base change formula shows that, for a given constructible
$F$ in $D(X)$ and a morphism of finite presentation $f:X\to Y$ over some base $S$,
one can choose the open subset $U\subset S$ so that the conclusion of the
generic base change theorem holds for $f_*(F)$ in $D$ and so that, for any morphism
of six-functors formalisms $\varphi\colon D\to D'$ (with $D'$ satisfying the
assumption of \ref{def:6FF} as well), the canonical map
\[
\varphi(f_*(F))\to f_*(\varphi(F))
\]
becomes an isomorphism over $U$ and the conclusion of the generic base change theorem
holds for $\varphi(F)$ with respect to $f$ in $D'$ (the idea is that we can do the proof of the generic base change formula for $F$ in $D$ and for $\varphi(F)$ in $D'$
in parallel and observe that $\varphi$ is compatible with
each reduction step, thus giving rise to the same choice of locus $U$ over which
the base change formula holds). 
\end{rem}
\begin{cor}\label{cor:gbc}
If moreover $F$ is strictly of geometric
origin, then the open subscheme $U\subset S$ in the statement
of Theorem \ref{thm:gbc} may be chosen indepently of $D$.
\end{cor}
\begin{proof}
Let $M$ in $\DM_c(X)$ such that $F\cong r_D(M)$ in $D(X)$.
Apply the generic base change formula to $M$, which gives an open subscheme $U\subset S$
such that the conclusion of Theorem \ref{thm:gbc} holds in $\DM$.
The conclusion of Theorem \ref{thm:gbc} holds for $r_D(M)$
in $D$ as well, by the previous Remark.\footnote{For the reader
who does not want to verify the claims of Remark \ref{rem:gbc},
Corollary \ref{cor:gbc}
can be proved as follows in the case where $D$ behaves well:
the conclusion
of Theorem \ref{thm:gbc} holds for $F=r_D(M)$ in $D$
with respect to the motivically
chosen $U\subset S$: since $r_D\colon\DM\to D$ commutes with the
six operations, all the base change isomorphisms involving $M$ in $\DM$ will yield
analogous base change isomorphisms involving $F$ in $D$.}
\end{proof}
\begin{paragr}
We fix a six-functors formalism $D$ defined over $S$-schemes,
as in paragraph \ref{def:6FF}. Let $E$ be a vector bundle over $S$
and $\PP=\PP(E)$ the associated projective bundle over $S$.
We denote by $\PP^\wedge$ the projective bundle associated to the dual of $E$,
and we assume given a diagram of $S$-schemes of finite presentation
\[
X\overset{e}{\to}\bar X\overset{\bar f}{\to} \PP
\]
(in practice, $e$ will be an open immersion and $\bar X$ will be a compactification
of $X$ as an $S$-scheme, but, for the moment, $e$ and $\bar f$ could be any maps).
We let
\[
\huniv\colon\HH\hookrightarrow\PP\times_S\PP^\wedge
\]
be the universal hyperplane in $\PP$, and write
\[
\PP\overset{\pi}{\leftarrow}\PP\times_S\PP^\wedge\overset{\pi^\wedge}{\to}\PP^\wedge
\]
for the canonical projections. We consider the following diagram, in which all
squares are cartesian.
\[
\begin{tikzcd}
H_X\ar{r}{\eta}\ar[hook]{d}{h_X}&
H_{\bar{X}}\ar{r}\ar[hook]{d}{h_{\bar X}}&\HH\ar[hook]{d}{\huniv}&\\
X\times_S\PP^\wedge\ar{r}{\gamma}\ar{d}{p}& \bar X\times_S\PP^\wedge\ar{r}\ar{d}{\bar p}&\PP\times_S\PP^\wedge\ar{r}{\pi^\wedge}\ar{d}{\pi}&\PP^\wedge\ar{d}\\
X\ar{r}{e}&\bar X\ar{r}&\PP\ar{r}&S
\end{tikzcd}
\]
Observe that both projections $\HH\to\PP$ and $\HH\to\PP^\wedge$
are smooth. Given a constructible object $F$ in $D(X)$,
we then have isomorphisms
\[
\bar{p}^*\, e_*(F)\cong \gamma_*\, p^*(F)
\]
and
\[
h^*_{\bar X}\, \gamma_*\, p^*(F)\cong \eta_*\, h^*_X\, p^*(F)\, .
\]
The first one is an instance of the smooth base change isomorphism,
whereas the second one comes from the following sequence of isomorphisms
(using the smoothness of $\HH$ over the projective bundle $\PP$):
\begin{align*}
h^*_{\bar X}\, \gamma_*\, p^*(F)&\cong h^*_{\bar X}\, \bar{p}^*\, e_*(F)\\
&\cong (\bar{p}h_{\bar{X}})^*\, e_*(F)\\
&\cong \eta_*\, (p h_X)^*(F)\\
&\cong \eta_*\, h^*_X\, p^*(F)\, .
\end{align*}
This leads to the following statement that goes back to
Beilinson~\cite{Beilinson} and Katz~\cite{Katz1}.
\end{paragr}
\begin{thm}[shadow of the weak Lefschetz theorem]\label{thm:weak Lefschetz}
Under the assumptions above, there exists a dense open subscheme $U\subset\PP^\wedge$
with the following property: for any hyperplane $H_t$ in $\PP$
defined over an $S$-scheme $T$
corresponding to a map $t\colon T\to U$ over $S$, if we form the corresponding
diagram of cartesian squares
\[
\begin{tikzcd}
X_T - H\ar[hook]{r}{j}\ar{d} & X_T\ar{d}{e_T}& H\ar{d}{u}\ar[hook']{l}[swap]{i}\\
\bar X_T - \bar H\ar[hook]{r}{\bar \jmath}
& \bar X_T & \bar H \ar[hook']{l}[swap]{\bar\imath}
\end{tikzcd}
\]
in which $X_T$ and $\bar X_T$ are the pullbacks of $X$ and $\bar X$ over $T\to S$
respectively,
and $\bar H$ is the pullback of $H_t$ along the projection
$H_{\bar X}\to\PP^\wedge$, we have a canonical isomorphism
\[
\bar \jmath_!\, \bar \jmath^*\, (e_T)_*(F_T)\cong (e_T)_*\, j_!\, j^*(F_T)
\]
with $F_T$ the pullback of $F$ along the projection $X_T\to X$.
Moreover, if ever $F$ is strictly of geometric origin, the subscheme $U$
can be chosen independently of $D$.
\end{thm}
\begin{proof}
We apply the generic base change formula to the restriction of $p^*(F)$ on $H_X$
and $\eta$, seen
over the scheme of hyperplane sections $\PP^\wedge$. There is thus a
dense open subscheme $U\subset\PP^\wedge$ over which $\eta_*\, h^*_X\, p^*(F)$
is constructible and its formation stable under any base change along any map
$T\to U\subset\PP^\wedge$. Applying the generic base change formula to $p^*(F)$
with respect to the map $\gamma$ over $\PP^\wedge$,
we may shrink $U$ so that the formation of $\gamma_*\, p^*(F)$ is
stable under any base change along any map $T\to U\subset\PP^\wedge$ as well.

Let us consider an $S$-scheme $T$ together with a map $t\colon T\to U$
defining an hyperplane $H_t$ in $\PP$. We obtain pullback
squares of the form
\[
\begin{tikzcd}
H \ar{r}{t_{H}}\ar{d}[swap]{i}&H_{X}\ar{d}{h_{X}}\\
X_T\ar{r}{t_{X}}&X\times_S\PP^\wedge
\end{tikzcd}
\quad\text{and}\quad
\begin{tikzcd}
\bar H \ar{r}{t_{\bar H}}\ar{d}[swap]{\bar\imath}&H_{\bar X}\ar{d}{h_{\bar X}}\\
\bar X_T\ar{r}{t_{\bar X}}&\bar X\times_S\PP^\wedge
\end{tikzcd}
\]
inducing base change isomorphisms
\[
t^*_{\bar X}\, \gamma_*\, p^*(F)\cong (e_T)_*\, \bar\imath^*\, p^*(F)= (e_T)_*(F_T)
\]
and
\[
t^*_{\bar H}\, \eta_*\, h^*_X\, p^*(F)\cong u_*\, t^*_H\, p^*(F)=u_*\, i^*(F_T)\, .
\]
We claim that these induce a canonical isomorphism:
\[
\bar\imath_*\, \bar\imath^*\, (e_T)_*(F_T)\cong (e_T)_*\, i_* \, i^*(F_T)\, .
\]
Indeed, we have:
\begin{align*}
\bar\imath_*\, \bar\imath^*\, (e_T)_*(F_T)
&\cong \bar\imath_*\, \bar\imath^*\, t^*_{\bar X}\, \gamma_*\, p^*(F)\\
&\cong \bar\imath_*\, t^*_{\bar H}\, h^*_{\bar X}\, \gamma_*\, p^*(F)\\
&\cong \bar\imath_*\, t^*_{\bar H}\, \eta_*\, h^*_X\, p^*(F)\\
&\cong \bar\imath_*\, u_*\, i^*(F_T)\\
&\cong (e_T)_*\, i_*\, i^*(F_T)\, .
\end{align*}

On the other hand, we have a canonical morphism of cofiber sequences:
\[
\begin{tikzcd}
\bar \jmath_!\, \bar \jmath^* (e_T)_*(F_T)\ar{r}\ar{d}&(e_T)_*(F_T)\ar[equals]{d}\ar{r}&
\bar \imath_*\, \bar \imath^*\, (e_T)_*(F_T)\ar{d}{\cong}\\
(e_T)_*\, j_!\, j^*(F_T)\ar{r}&(e_T)_*(F_T)\ar{r}&(e_T)_*\, i_*\, i^*(F_T)
\end{tikzcd}
\]
in which the vertical map on the right hand side is an isomorphism by the
computation we just did. Since the middle vertical map is the identity,
this proves that the vertical map on the left hand side is an isomorphism.
The fact that the open subscheme $U\subset\PP^\wedge$ may be chosen
independently of $D$
if ever the sheaf $F$ is strictly of geometric origin follows from
Corollary~\ref{cor:gbc}.
\end{proof}
\begin{cor}\label{cor:abstract weak Lefschetz}
Let $k$ be an infinite field, and $X$ a quasi-projective scheme over $k$.
We choose an open immersion $e:X\to\bar X$ into a projective $k$-scheme,
and a constructible motivic sheaf $M$ in $\DM(X)$.
There exists a dense open subset $U$ (hence with at least one rational point)
of the family of hyperplane sections of $\bar X$,
such that, for any $\bar H$ in $U$, with $H=\bar H\cap X$ and
with corresponding open complements
$\bar \jmath\colon \bar X-\bar H\to\bar X$ and $j\colon X-H\to X$,
and for any six functors formalism $D$ as in paragraph \ref{def:6FF},
if we put $F=r_D(M)$ for the realization of $M$ in $D$, then there is a canonical
isomorphism
\[
\bar \jmath_!\, \bar \jmath^*\, e_*(F)\cong e_*\, j_!\, j^*(F)\, .
\]
In particular, the cohomology with compact support of $\bar X -\bar H$
with coefficients in $e_*(F)$ coincides with the cohomology of $X$ with
coefficients in $j_!\, j^*(F)$.
\end{cor}
\begin{paragr}\label{def:good position}
A fundamental situation that will come back eventually in these notes
is the following one.
Let $F$ be in $D(X)$ for an affine scheme of finite type $X$ over a field $k$.
We assume chosen an affine open embedding $e\colon X\to \bar X$ into a projective
scheme. A hyperplane section $\bar H$ of $\bar X$ with trace $H=X\cap \bar H$
is said to be \emph{transverse to $F$} if
by forming the corresponding pullback squares below
\[
\begin{tikzcd}
X - H\ar[hook]{r}{j}\ar[hook]{d} & X\ar[hook]{d}{e}
& H\ar[hook]{d}{u}\ar[hook']{l}[swap]{i}\\
\bar X - \bar H\ar[hook]{r}{\bar \jmath}
& \bar X & \bar H \ar[hook']{l}[swap]{\bar\imath}
\end{tikzcd}
\]
there is a canonical isomorphism
\[
\bar \jmath_!\, \bar \jmath^*\, e_*(F)\cong e_*\, j_!\, j^*(F)\, .
\]
Theorem \ref{thm:weak Lefschetz}
tells that there is a dense open subspace of such hyperplane
sections and therefore that, possibly after base change along a finite
extension of $k$, one can provide such nice hyperplane sections.
Furthemore, if ever $F$ is strictly of geometric origin, one can choose
a dense open subspace of hyperplane sections in good position with respect to $F$
independently of $D$.
\end{paragr}
\section{Perverse sheaves}
\begin{paragr}\label{def:t-struct}
Let $D$ be a six-functors formalism defined on schemes of finite presentation
over a base $S$. A \emph{$t$-structure} on $D$ is a family of $t$-structures
on $D(X)$ for each $X$ with the following properties
(we adopt cohomological conventions):
\begin{itemize}
\item[(a)] for any scheme $X$ any integer $n\in\ZZ$,
the Tate object $\unit_X(n)$ is in the heart of $D(X)$;
\item[(b)] for any scheme $X$, the tensor product functor preserves objects concentrated
in degree $\leq 0$ in the sense of the $t$-structure (i.e. tensoring with an object of the heart is right $t$-exact in $D(X)$);
\item[(c)] for any map $f\colon X\to Y$, the functor $f^*\colon D(Y)\to D(X)$
is $t$-exact;
\item[(d)] for any separated morphism of finite presentation $f:X\to Y$, the
functor $f_!\colon D(X)\to D(Y)$ is right $t$-exact;
\item[(e)] the truncation functors preserve dualizable objects and all dualizable
objects of $D(X)$ are bounded for the $t$-structure;
\item[(f)] producing the dual of dualizable objects $M\mapsto M^\wedge$
is a $t$-exact functor.
\end{itemize}
As a consequence, Tate twists $M\mapsto M(n)$ are $t$-exact functors.
Any constructible object $F$ of $D(X)$ is bounded.
Indeed, if  we define $f$ as the coproduct of the embeddings of the
strata over which $F$ becomes dualizable, then $f^*(F)$ is dualizable in $D(X')$, and
since $f^*$ is conservative (by localization) and $t$-exact, this proves
the claim. Remark that condition (f) is only reasonnable with $\FF$-linear
coefficients for some field $\FF$.\footnote{Condition (f)
can be avoided but we would need two $t$-structures
whose restrictions to dualizable objects would be exchanged through ordinary
duality, and that would only make notations worse.}
We denote by $D^{\leq i}_c(X)$ (by $D^{\geq i}_c(X)$ respectively)
the full subcategory of $D_c(X)$ that consists of constructible objects
with cohomology concentrated in degrees $\leq i$ (in degrees $\geq i$, respectively)
with respect to the $t$-structure above. Such a $t$-structure will be called
\emph{ordinary}.
\end{paragr}
\begin{example}
It is well known that, to say the least,
the existence of such a $t$-structure is very conjectural as far as $\DM$ is concerned.
However, $D(X)=D(X,\QQ_\ell)$, $D(X)=D(X,\FF_\ell)$
and $D(X)=D(X,\FF_\mathcal{U})$ have canonical
$t$-structures coming from the usual $t$-structures on the derived category
of sheaves of $\Lambda$-modules on the small \'etale site of $X$, with $\Lambda$
and finite commutative ring. If moreover $X$ is noetherian, all the properties listed
above will be fulfilled in these cases.
\end{example}
\begin{paragr}\label{def:perverse t-structure}
Let $k$ be a field and $D$ a six-functors formalism as in \ref{def:6FF}
defined on schemes of finite type over $k$ and equipped with a $t$-structure
in the sense of \ref{def:t-struct}.
We can then reproduce word for word
section 2.2 of \cite{BBDG}. We will always consider here the autodual perversity.
The perverse $t$-structure on $D(X)$ is characterized by noetherian induction,
using the following facts:
\begin{itemize}
\item[(i)] The perverse $t$-structure coincides with the usual one on $D(X)$
for any $k$-scheme $X$ of dimension $\leq 0$;
\item[(ii)] For any closed immersion $i\colon Z\to X$, the functor
$i_!\colon D(Z)\to D(X)$ is $t$-exact with respect to the perverse $t$-structure,
and for any open immersion $j\colon U\to X$, the functor
$j^*\colon D(X)\to D(U)$ is $t$-exact with respect to the perverse $t$-structure;
\item[(iii)] For any regular scheme $X$ of pure dimension $d$, any
dualizable object $F$ of $D(X)$ that is also in the heart of the
usual $t$-strucure
defines an object $F[-\dim(X)]$ in the
heart of the perverse $t$-structure of $D_c(X)$.
\end{itemize}
More precisely, we define the perverse $t$-structure on $D(X)$ by induction on the
dimension of $X$ (together with properties (i), (ii) and (iii) above).
For $X$ of dimension $\leq 0$, this is the ordinary $t$-structure
(the one from \ref{def:t-struct}). For $X$ of dimension $d>0$, assuming that the
perverse $t$-structure is defined on $D(Z)$ for any scheme $Z$ of dimension $<d$,
we define the perverse $t$-structure on $D(X)$ as follows. For any nowhere dense
closed subscheme $i\colon Z\hookrightarrow X$ with open complement
$j\colon U\hookrightarrow X$ such that $U_\mathit{red}$ is regular,
we denote by $D_U(X)$ the full subcategory of $D(X)$ spanned by constructible
objects $F$ with the property that $j^*(F)$ is dualizable.
If $D^\wedge(U)$ denotes the full subcategory of dualizable objects of $D(U)$,
the functors $i_*$ and $j^*$ expressing $D(X)$ as a recollement of $D(U)$ and $D(Z)$
in the sense of \cite[1.4.3]{BBDG} exhibit $D_U(X)$ as a recollement of $D^\wedge(U)$
and $D(Z)$. If $U=\coprod_i U_i$ is the decomposition of $U$ into its
irreducible components, then
\[
D^\wedge(U)\cong\prod_i D^\wedge(U_i)
\]
and we define the perverse $t$-structure on $D^\wedge(U)$ as the product perverse
$t$-structure, where the perverse $t$-structure on each $D^\wedge (U_i)$
is defined by the property that its heart consists of those objects of the form
$F[-\dim(U_i)]$ with $F$ in the heart of the ordinary $t$-structure on $D(U_i)$.
We have a perverse $t$-structure on $D(Z)$ by induction, and thus a perverse
$t$-structure on $D_U(X)$, by \cite[Th\'eor\`eme 1.4.10]{BBDG}.
If $Z'$ is a nowhere dense open subscheme of $X$ containing $Z$ and if $U'$ is the complement of $Z'$, since the extension by zero functor $D(Z)\to D(Z')$
is $t$-exact (by induction on dimension, property (ii) holds true),
and since the restriction functor $D^\wedge(U)\to D^\wedge(U')$ is
obviously $t$-exact, we see that the inclusion functor
\[
D_U(X)\hookrightarrow D_{U'}(X)
\]
is $t$-exact for the perverse $t$-structures.
Writing
\[
\varinjlim_U D_U(X)=D_c(X)
\]
where the filtered colimit is indexed by dense open subschemes $U\subset X$
such that $U_\mathit{red}$ is regular thus defines a $t$-structure on $D_c(X)$.
It is clear that properties (i), (ii) and (iii) are true by construction.

We denote by $^pD^{\leq i}_c(X)$ (by $^pD^{\geq i}_c(X)$ respectively)
the full subcategory of $D_c(X)$ that consists of constructible objects
with cohomology concentrated in degrees $\leq i$ (in degrees $\geq i$, respectively)
with respect to the perverse $t$-structure defined above. 

An object of $D(X)$ is called \emph{perverse} if it is in the heart of the perverse
$t$-structure on $D_c(X)$.
\end{paragr}
\begin{lemma}\label{lemma:restr. right t-exact}
For any closed immersion $i\colon Z\to X$, the functor $i^*\colon D(X)\to D(Z)$
is right $t$-exact with respect to the perverse $t$-structure.
\end{lemma}
\begin{proof}
This is obvious since its right adjoint $i_*$ is (left) $t$-exact
with respect to the perverse $t$-structure.
\end{proof}
\begin{paragr}
Observe that, by generic base change, the functors $f_*$ preserve constructible
objects for any morphism $f$ between $k$-schemes of finite type.
This implies that, for any separated morphism of finite type $f\colon X\to Y$,
the functor $f^!$ preserves constructible objects in $D$: since this is a Zariski-local
property, it suffices to check this property for $f$ quasi-projective.
By relative purity, for $f$ smooth, we have $f^!(\unit)\otimes f^*\cong f^!$
with the object $f^!(\unit)$ $\otimes$-invertible (hence dualizable and therefore
constructible), so that it suffices to check that $i^!$ preserves
constructible objects for $i\colon Z\to X$ a closed immersion.
If $j\colon U\to X$ is the open complement, the fiber sequence
\[
i^!\to i^*\to i^*\, j_*\, j^*
\]
shows that $i^!$ preserves constructible objects.
We also see that the internal Hom preserve constructible objects:
For $F$ and $G$ constructible in $D(X)$, we choose an open
immersion $j\colon U\to X$ so that $j^*(F)$ is dualizable.
The internal Hom $\uHom(F,G)$ then fits into a fiber sequence
\[
i_*\, i^!\uHom(F,G)\to \uHom(F,G)\to j_*\, j^*\uHom(F,G)
\]
that is isomorphic to
\[
i_*\uHom(i^*(F),i^!(G))\to\uHom(F,G)\to j_*\uHom(j^*(F),j^*(G))\, .
\]
We see that $\uHom(i^*(F),i^!(G))$ by induction on the Krull dimension
whereas the object $\uHom(j^*(F),j^*(G))$ is the tensor product of the dual of the dualizable
object $j^*(F)$ with the constructible object $j^*(G)$, hence is constructible.
We can thus define Verdier duality in $D$ as the functor
\[
D_X\colon D_c(X)^{\mathit{op}}\to D_c(X)\ , \quad F\mapsto\uHom(F,a^!(\unit))
\]
where $a\colon X\to\spec k$ is the structural map.
Observe that, since
\[
f^!\uHom(A,B)\cong\uHom(f^*(A),f^!(B))
\]
and
\[
\uHom(f_!(A),B)\cong f_*\uHom(A,f^!(B))\, ,
\]
we see that Verdier duality commutes
with $f_*$ for $f$ proper.
\end{paragr}
\begin{rem}\label{rem:verdier dualizable}
On regular schemes, Verdier duality preserves dualizable objects:
For $F$ dualizable in $D(X)$ with $X_\mathit{red}$ regular,
we have $D_X(\unit_X)=\unit_X(d)[2d]$ with $d=\dim X$ (it suffices to prove this
for $D=DM$ since $r_D\colon DM\to D$ preserves the six operations) so that
\[
D_X(F)\cong F^\wedge\otimes\unit_X(d)[2d]=F^\wedge(d)[2d]
\]
with $F^\wedge$ the ordinary dual of $F$.
In particular, we have
\[
D_X(F[d]])\cong D_X(F)[-d]\cong F^\wedge[d](d)\, .
\]
\end{rem}
\begin{defn}\label{def:Artin-Verdier}
An \emph{Artin-Verdier six-functors formalism} is a six-functors formalism
$D$ defined for schemes of finite type over a field $k$ as in \ref{def:6FF}
equipped with an ordinary $t$-structure in the sense of \ref{def:t-struct}
with the following two extra properties:
\begin{itemize}
\item[(i)] Verdier duality holds in $D$: for any scheme of finite type $X$ over $k$,
Verdier duality induces an equivalence
\[
D_X\colon D_c(X)^\mathit{op}\cong D_c(X)\, .
\]
\item[(ii)] Artin vanishing holds in $D$: for any affine morphism of $k$-schemes
of finite type $f\colon X\to Y$
the induced functor $f_*\colon D_c(X)\to D_c(Y)$ is right $t$-exact with respect to
the perverse $t$-structure.
\end{itemize}
\end{defn}
\begin{example}
$D(-,\FF)$ with $\FF$ a finite field,
$D(-,\QQ_\ell)$ with $\ell$ any prime number,
and $D(-,\FF_\mathcal{U})$ with $\mathcal{U}$ any ultrafilter on the set of prime numbers, as defined in \ref{ex:ladic} and \ref{ex:ultra} are Artin-Verdier six-functors
formalisms -- properties (i) and (ii) are even known to hold over any quasi-excellent base, after the work of
Gabber \cite[Expos\'e XVII, Th\'eor\`eme 0.2 and Expos\'e XV, Th\'eor\`eme 1.1.2]{gabber3}.
Other examples are the de Rham six-functors formalism, the $\FF$-linear
version of Betti six-functors formalism (for any field of coefficient $\FF$),
and the rigid six-functors formalism.
\end{example}
\begin{paragr}
From now on, we fix an Artin-Verdier six-functors formalism $D$.
We consider the associated perverse $t$-structure on $D$ (see \ref{def:perverse t-structure}).
\end{paragr}
\begin{rem}\label{rem:Verdier restriction}
We have formulas of the form
\[
f^!\, D_X\cong D_Y\, f^*\quad\text{and}\quad D_Y\, f_!\cong f_*\, D_X
\]
for any morphism of $k$-schemes of finite type $f\colon X\to Y$.
Since Verdier duality is an equivalence, this implies dually that
\[
f^*\, D_X\cong D_Y\, f^!\quad\text{and}\quad D_Y\, f_*\cong f_!\, D_X
\]
holds as well.
Given an immersion $\imath\colon S\to X$ with
$S_\mathit{red}$ regular, and $F$ constructible in $D(X)$ such that
$\imath^*(F)$ (with $\imath^!(F)$, respectively)
is dualizable, we observe that
\[
D_S\,\imath^*(F)\cong\imath^!\, D_S(F)
\quad\text{($D_S\,\imath^!(F)\cong\imath^*\, D_S(F)$, respectively)}
\]
is dualizable as well: it is the tensor product of the $\otimes$-invertible
object $D_S(\unit_S)\cong\unit_S(\dim S)[2\dim S]$ with the
ordinary dual of $i^*(F)$ (of $i^!(F)$, respectively);
see \ref{rem:verdier dualizable}.
\end{rem}
\begin{prop}\label{prop:Verdier perverse exact}
Verdier duality is $t$-exact with respect to the perverse $t$-structure.
\end{prop}
\begin{proof}
We proceed by induction on the Krull dimension $d$ of $X$. For $d\leq 0$,
all constructible objects are dualizable and Verdier duality coincides with ordinary
duality, and the bidual of a dualizable object $F$ is canonically isomorphic to $F$.
Let us assume that $d>0$.
We choose a nowhere dense closed subscheme $i\colon Z\to X$ with open complement
$j\colon U\to X$ with $U_\mathit{red}$ regular.
We consider the category $D_U(X)$ of constructible objects $F$ in $D(X)$
such that $j^*(F)$ is dualizable.
We observe that Verdier duality induces a
$t$-exact functor with respect to the perverse $t$-structure
\[
D_c(Z)^\mathit{op}\to D_c(Z)
\]
by induction on the Krull dimension, and an equivalence
\[
D_U(U)^\mathit{op}\cong D_U(U)
\]
because, up to a Tate twist,
Verdier duality on $D_U(U)$ simply is ordinary duality
restricted to dualizable objects; see Remark \ref{rem:verdier dualizable}.
In fact, the perverse $t$-structure on $D_U(U)$
is the ordinary $t$-structure shifted by $\dim U$:
the $t$-structure whose heart is
\[
D^{\leq \dim U}_U(U)\cap D^{\geq \dim U}_U(U)\, .
\]
Verdier duality induces an equivalence
$D_U(U)^\mathit{op}\cong D_U(U)$
that is $t$-exact with respect to the perverse $t$-structure
because of formula $D_X(F[d]])\cong F^\wedge[d](d)$
for any $F$ in the ordinary heart of $D_U(U)$.
Verdier duality is $t$-exact with respect to the perverse $t$-structure
on $D(Z)$ by induction on $d$.
Verdier duality restricts to $D_U(X)$, and since the
inclusion $D_U(X)\hookrightarrow D(X)$ is $t$-exact with respect to the perverse $t$-structure by construction, it suffices to prove that Verdier
duality is $t$-exact on $D_U(X)$. In other words,
we want to prove that an object $F$ of $D_U(X)$ is in $^pD^{\leq 0}_U(X)$
if and only if is Verdier dual $D_X(F)$ is in $^pD^{\geq 0}_U(X)$.
By definition of glued $t$-structures, an object
$F$ is in $^pD^{\leq 0}_U(X)$ if only if $i^*(F)$ lies in
$^pD^{\leq 0}(Z)$ and $j^*(F)$ lies in $^pD^{\leq 0}_U(U)$.
Dually, an object $F$ is in $^pD^{\leq 0}_U(X)$ if only if $i^!(F)$ lies in
$^pD^{\geq 0}(Z)$ and $j^*(F)$ lies in $^pD^{\leq 0}_U(U)$.
Since Verdier duality exchanges $i^*$ and $i^!$ and
commutes with $j^*$, this achieves the proof.
\end{proof}
\begin{cor}\label{cor:compact supprt perverse left exact}
For any affine morphism of $k$-schemes
of finite type $f\colon X\to Y$
the induced functor $f_!\colon D_c(X)\to D_c(Y)$ is left $t$-exact with respect to
the perverse $t$-structure.
\end{cor}
\begin{cor}
For any finite morphism $f\colon X\to Y$, the functor
$f_!=f_*\colon D(X)\to D(Y)$ is $t$-exact with respect to
the perverse $t$-structure.
\end{cor}
\begin{lemma}\label{lemma:dualizable hearts}
Let $X$ be of pure dimension $d$,
with $X_\mathit{red}$ regular and $F$ dualizable in $D(X)$.
Then $F$ is perverse (is in $^pD^{\leq 0}(X)$)
if and only if $F[-d]$ is in the heart
of the ordinary $t$-structure (is in $D^{\leq 0}(X)$, respectively).
\end{lemma}
\begin{proof}
We will prove the respective case first.
Let us assume that $F$ is in $^pD^{\leq 0}(X)$
but not in $^pD^{< 0}(X)$. There is smallest
integer $n$ such that $F$ lies in $D^{\leq n}_c(X)$.
Then $F[-n]$ is in $D^{\leq 0}_c(X)$ and thus $F[d-n]$
is in $^pD^{\leq 0}(X)$, which yields $n\leq d$.

Let us assume now that $F$ is perverse.
Since $F$ is dualizable and $X_\mathit{red}$ regular of dimension $d$,
we have $D_X(F)\cong F^\wedge(d)[2d]$ (Remark \ref{rem:verdier dualizable}).
By virtue of
Proposition \ref{prop:Verdier perverse exact},
this means that
\[
F^\wedge[2d]=(F[-d])^\wedge[d]
\]
is dualizable and perverse as well.
This means that both $F[-d]$ and $(F[-d])^\wedge$
are dualizable and in $D^{\leq 0}(X)$. Since ordinary duality
of dualizable objects is $t$-exact with respect to the
ordinary $t$-structure, this implies that $F[-d]$
lies in the heart of the ordinary $t$-structure.
\end{proof}
\begin{prop}\label{prop:charact. perverse t-struct}
Let $F$ be a constructible object of $D(X)$ and $n\in\ZZ$ an integer.
The following conditions are equivalent.
\begin{itemize}
\item[(i)] $F$ belongs to $^pD^{\leq 0}_c(X)$ (to $^pD^{\geq 0}_c(X)$, respectively);
\item[(ii)] For any irreducible closed subscheme $T\subset X$, there is
a dense open subscheme $S$ in $T$ inducing an embedding $\imath_S\colon S\to X$
such that $\imath^*_S(F)$ belongs to $D^{\leq \dim S}_c(S)$
(such that $\imath^!_S(F)$ belongs to $D^{\geq \dim S}_c(S)$, respectively).
\item[(iii)] For any finite stratification $\mathcal{S}$ of $X$
by strata that are irreducible with regular reduction,
and for any embedding $\imath_S\colon S\to X$ with $S\in\mathcal{S}$,
assuming that both $i^*_S(F)$ and $i^!_S(F)$ are dualizable,
we have $\imath^*_S(F)$ in $D^{\leq \dim S}_c(S)$
(we have $\imath^!_S(F)$ in $D^{\geq \dim S}_c(S)$, respectively).
\item[(iv)] There exists a finite stratification $\mathcal{S}$ of $X$
by strata that are irreducible with regular reduction,
such that, for any embedding $\imath_S\colon S\to X$ with $S\in\mathcal{S}$,
both $i^*_S(F)$ and $i^!_S(F)$ are dualizable and
we have $\imath^*_S(F)$ in $D^{\leq \dim S}_c(S)$
(we have $\imath^!_S(F)$ in $D^{\geq \dim S}_c(S)$, respectively).
\end{itemize}
\end{prop}
\begin{proof}
Let $Y$ be a scheme of finite type over $k$.
For any constructible object $K$ in $D(Y)$, there is a dense open subscheme
$U$ of $Y$ such that the restriction of $K$ on $U$ is dualizable in $D(U)$.
Since any scheme of finite over scheme is excellent, there is a dense
subscheme $U$ of $Y$ that is a finite disjoint uinion of irreducible schemes
whose reductions are regular. Therefore, there exists
at least one finite stratification
$\mathcal{S}$ of $X$
by strata that are irreducible with regular reduction,
such that, for any embedding $\imath_S\colon S\to X$ with $S\in\mathcal{S}$,
both $i^*_S(F)$ and $i^!(S)$ are dualizable.
The fact that we have equivalences (i)--(iv) in the case of
$^pD^{\leq 0}_c(X)$ follows right away from
Lemmata \ref{lemma:restr. right t-exact} and \ref{lemma:dualizable hearts}.
Since Verdier duality exchanges $^pD^{\leq 0}_c(X)$ and $^pD^{\geq 0}_c(X)$
(because it is $t$-exact), the respective case follows now from
the last part of Remark \ref{rem:Verdier restriction}.
\end{proof}
\begin{cor}
We have the following exactness properties with respect to the perverse $t$-structure.
\begin{itemize}
\item[(a)] For any quasi-finite morphism $f\colon X\to Y$, the
functor $f^*\colon D_c(Y)\to D_c(X)$ is right $t$-exact and the
exceptional pullback functor $f^!\colon D_c(Y)\to D_c(X)$ is left $t$-exact.
\item[(b)] For any quasi-finite morphism $f\colon X\to Y$, the
functor $f_!\colon D_c(Y)\to D_c(X)$ is right $t$-exact and
$f_*\colon D_c(Y)\to D_c(X)$ is left $t$-exact.
\end{itemize}
\end{cor}
\begin{proof}
To prove (a), we use criterium (ii)
of Proposition \ref{prop:charact. perverse t-struct}.
Let $T$ be a closed subscheme of $X$.
There is a dense open subscheme $U$ of the closure of
the image of $T$ in $Y$ such that
$\imath_U^*(F)$ belongs to $D^{\leq\dim U}_c(U)$.
Therefore, if $S$ denotes the pullback of $U$ in $T$,
we have $\imath_U^*\, f^*(F)$ in $D^{\leq\dim U}_c(S)$.
Since $\dim S=\dim U$ (because $f$ is quasi-finite
and $S$ dominates $U$ by construction),
this proves that $f^*(F)$ is in $^pD^{\leq 0}(X)$.
By Verdier duality, this proves that $f^!$ is left $t$-exact.
Property (b) follows straight away from property (a).
\end{proof}
\begin{cor}\label{cor:etale perverse exact}
If $f\colon X\to Y$ is \'etale, then $f^*$
is $t$-exact with respect to the
perverse $t$-structure.
\end{cor}
\begin{cor}\label{cor:quasi-finite perverse exact}
If $f\colon X\to Y$ is both affine and quasi-finite, then
both functors $f_!$ and $f_*$ are $t$-exact with respect to the
perverse $t$-structure.
\end{cor}
\begin{prop}
If $f\colon X\to Y$ is a finite flat morphism with both $X_\mathit{red}$
and $Y_\mathit{red}$ regular, then, for any $F$
that is both perverse and dualizable in $D(Y)$, there is
an isomorphism $f^!(F)\cong f^*(F)$ and $f^*(F)$ is both perverse
and dualizable in $D(X)$.
\end{prop}
\begin{proof}
There is always a canonical morphism
\[
f^!(\unit_Y)\otimes f^*(F)\to f^!(F)\, .
\]
The latter is an isomorphism for $F$ dualizable (this is a nice exercise
that does not use any property of $f$) and $f^!(\unit_Y)\cong\unit_X$
by relative purity (this times using that both $X$ and $Y$ are regular
up to reduction and of same dimension). If $F$ is dualizable in $D(Y)$,
since $Y_\mathit{red}$ is regular,
it is perverse if and only if $F[-\dim Y]$ is in the
heart of the ordinary $t$-structure (Lemma \ref{lemma:dualizable hearts}).
The functor $f^*$ preserves dualizable because it is symmetric monoidal.
Since $f^*$ is $t$-exact for the
ordinary $t$-structure, this proves the proposition.
\end{proof}

\begin{prop}\label{prop:basic vanishing}
Let $a\colon X\to Y$ be a flat affine morphism between
$k$-schemes of finite type, of relative dimension $d$.
\begin{enumerate}
\item[(a)] For any $F$ in $D^{\leq 0}_c(X)$, there is a dense open
subscheme $V$ of $Y$ such that the restriction
$a_*(F)_{|_V}$ belongs to $D^{\leq d}_c(V)$.
\item[(b)] For any $F$ in $D^{\geq 0}_c(X)$,
such that $F[\dim X]$ belongs to $^pD^{\geq 0}_c(X)$,
there is a dense open
subscheme $V$ of $Y$ such that the restriction
$a_!(F)_{|_V}$ belongs to $D^{\geq d}_c(V)$.
\end{enumerate}
Moreover, in both cases, if $F$ is strictly of geometric origin, then
the open subscheme $V$ may be chosen independently of $D$.
\end{prop}
\begin{proof}
For (a), We proceed by induction on $d$.
Since we are allowed to shrink $Y$ at will,
we may assume $Y_{\mathit{red}}$ to be regular.
Since $D$ is insensitive to nil-immersions, we may as well assume that $Y$
is regular.
For $d\leq 0$, this is obvious:
since $X$ is affine and quasi-finite over $Y$, we know that $a_*$
is $t$-exact by Corollary \ref{cor:quasi-finite perverse exact}.
For $d>0$, given $F$ constructible in $D^{\leq 0}(X)$,
we choose $i\colon Z\to X$ a nowhere dense closed immersion with
affine complement $j\colon U\to X$ such that $U_\mathit{red}$ is regular
and everywhere of dimension $\dim X$,
and $j^*F$ is dualizable in $D(U)$
(if $F$ is strictly of geometric origin we make this choice in $\DM$
to make this independent of $D$,
by Proposition \ref{prop:generic dualizability}). We then have
$j^*F[\dim U]$ in $^pD^{\leq 0}(U)$ hence $j_!\, j^*F[\dim U]$
is in $^pD^{\leq 0}(X)$.
Therefore, Artin vanishing tells us that
$a_*\, j_!\, j^*F[\dim U]$ belongs to $^pD^{\leq 0}(Y)$.
Shrinking $Y$, we may impose that $a_*\, j_!\, j^*F$
is dualizable (if $F$ is strictly of geometric origin, so is
$a_*\, j_!\, j^*F$ and we make this shrinking at the motivic level,
as in the proof of \ref{prop:generic dualizability}).
Since $d=\dim U - \dim Y$, by virtue of Lemma \ref{lemma:dualizable hearts},
this means that $a_*\, j_!\, j^*F[d]$ belongs to $D^{\leq 0}(Y)$,
hence $a_*\, j_!\, j^*F$ to $D^{\leq d}(Y)$.
By virtue of Grothendieck's generic flatness theorem,
shrinking $Y$, we may assume that $Z$ is flat over $Y$ as well
(this choice is purely geometric and does not depend on $F$ nor on $D$).
By induction, since the relative dimension of $Z$ over $Y$ is $< d$,
possibly after shrinking $Y$ one more time, we have
$(ai)_*\, i^*(F)$ in $D^{\leq d-1}(Y)$.
The cofiber sequence
\[
a_*\, j_!\, j^*F\to a_* F\to a_*\, i_*\, i^* F
\]
thus implies that $a_*(F)$ belongs to $D^{\leq d}(Y)$.

In order to prove (b), this is more
direct. Let $F$ be in $D^{\geq 0}_c(X)$, such that
$F[\dim X]$ belongs to $^pD^{\geq 0}_c(X)$.
Since $a_!$ is left $t$-exact with respect to the perverse
$t$-structure (Cor.~\ref{cor:compact supprt perverse left exact}),
we have $a_!(F)[\dim X]$ in $^pD^{\geq 0}_c(Y)$.
Shrinking $Y$, we may assume that $a_!(F)$ is
dualizable
and that $Y_{\mathit{red}}$ is regular,
which implies by Lemma \ref{lemma:dualizable hearts} that
\[
a_!(F)[\dim X - \dim Y]=a_!(F)[d]
\]
belongs to $D^{\geq 0}_c(Y)$.
\end{proof}

\begin{rem}\label{rem:perverse Artin}
Let $a\colon X\to\spec k$ be a separated morphism of finite type.
For any $F$ in $D(X)$, we can define
$\Gamma(X,F)=a_*(F)$ and $\Gamma_c(X,F)=a_!(F)$.
Taking cohomology in the sense of the ordinary $t$-structure,
we get objects
\[
H^i(X,F)=H^i(\Gamma(X,f))\quad\text{and}\quad H^i_c(X,F)=H^i(\Gamma_c(X,F))
\]
in the heart $D^\heartsuit(\spec k)=D^{\leq 0}(\spec k)\cap D^{\geq 0}(\spec k)$.

%
%
The following statement was
observed by Beilinson \cite{Beilinson}. Recall the notion
of hyperplane section transverse to an object of $D$,
as defined in \ref{def:good position}.
\end{rem}

\begin{thm}\label{thm:Beilinson}
Let $X$ be an affine scheme of finite type over $k$
and $e\colon X\to \bar X$ an affine open immersion into a projective
scheme. We consider given an affine open immersion $u\colon U\to X$
with $U_\mathit{red}$ regular.
We assume that $F$ is dualizable and in the heart of the ordinary $t$-structure
on $D_c(U)$. If $j\colon X-H\hookrightarrow X$
denotes the embedding of the complement of a hyperplane section that is
transverse to $u_!F$, then $H^q(X,j_!\, j^*u_!F)=0$
for any $q\neq \dim X$.
\end{thm}
\begin{proof}
Let $d=\dim X$.
Keeping track of the notations of \ref{def:good position},
the transversality assumption on $H$ implies that
we have canonical isomorphisms
\[
H^q_c(\bar X - \bar H, e_*\, u_!F)\cong H^q(X,j_!\, j^*\, u_!F)
\]
with both $X$ and $\bar X- \bar H$ affine of dimension $d$.
By virtue of Proposition \ref{prop:basic vanishing}~(a),
since all functors $j_!$, $j^*$ and $u_!$ are $t$-exact with
respect to the ordinary $t$-structure,
$H^q(X,j_!\, j^*\, u_!F)$ vanishes for $q>d$.
On the other hand, by virtue of Lemma \ref{lemma:dualizable hearts},
we see that $F[d]$ is perverse on $U$, and since
both $u$ and $e$ are affine immersions, Corollaries \ref{cor:etale perverse exact}
and \ref{cor:quasi-finite perverse exact} imply that
$e_*\, u_!F[d]$ is perverse
on $\bar X - \bar H$. Therefore, Proposition \ref{prop:basic vanishing}~(b)
implies that $H^q_c(\bar X - \bar H, e_*u_!F)$ vanishes for $q<d$.
\end{proof}
\section{Uniformly perverse motivic generators}
\begin{paragr}
Let $X$ be a scheme of finite type over a field $k$.
We denote by $M(X)$ the motive of $X$, i.e. the object of $\DM(\spec k)$
defined by
\[
M(X)=a_!\, a^!(\ZZ)
\]
with $a\colon X\to\spec k$ the structural morphism.\footnote{If we define $\DM(X)$
through $h$-motives, its is a $(\mathbf{P}^1,\infty)$-stabilization
of the $\mathbf{A}^1$-localization of the derived $\infty$-category of
sheaves of abelian groups on the big site of $k$ with respect to the $\h$-topology.
The object $M(X)$ is then what remains of the sheaf represented by $X$ through all this
process.} This construction defines a functor from the category of schemes of finite type over $k$ to the stable $\infty$-category $\DM(\spec k)$.
Given a closed subscheme $Z$ of $X$, we define $M(X,Z)$ through the
cofiber sequence
\[
M(Z)\to M(X)\to M(X,Z)
\]
For any object $A$ in $\DM(\spec k)$, we thus have a fiber sequence in the
derived category of abelian groups
\[
\RHom(M(X,Z),A)\to \RHom(M(X),A)\to \RHom(M(Z),A)
\]
with $\RHom$ the cochain complex of maps in $\DM(\spec k)$ (the latter is $\ZZ$-linear by
construction). Since $(a_!a^!(\ZZ))^\wedge\cong a_*a^*(\ZZ)$ by Verdier duality,
we observe that
\[
\RHom(M(X),A)\cong\RHom(\ZZ,a_*a^*(\ZZ))\cong \RHom(\ZZ_X,A_X)
\]
where third instance of $\RHom$ is the cochain complex of maps in $\DM(X)$,
$\ZZ_X=\unit_X$ is the $\otimes$-unit of $\DM(X)$, and $A_X=a^*(A)$ with $a$
the structural morphism of $X$ as above.
Therefore,
if $j\colon X-Z\hookrightarrow X$ is the open immersion
of the complement of $Z$ in $X$, and if $i\colon Z\hookrightarrow X$
is the embedding of $Z$
into $X$, the cofiber sequence
\[
j_!\, j^*(A_X)\to A_X\to i_*\, A_Z
\]
thus shows that we have a canonical isomorphism
\[
\RHom(M(X,Z),A)\cong \RHom(\ZZ_X,j_!\, j^*A_X)\, .
\]
Let $h:H\hookrightarrow X$ be another closed embedding.
We have a cocartesian square
\[
\begin{tikzcd}
M(Z\cap H)\ar{r}\ar{d}& M(Z)\ar{d}\\
M(H)\ar{r}&M(H\cup Z)
\end{tikzcd}
\]
hence cartesian squares of the form below.
\[
\begin{tikzcd}
\RHom(M(X,Z\cup H),A)\ar{r}\ar{d}&\RHom(M(X,Z),A)\ar{d}\\
\RHom(M(X,H),A)\ar{r}&\RHom(M(X,H\cap Z),A)
\end{tikzcd}
\]
Those correspond to cartesian squares
\[
\begin{tikzcd}
\RHom(\ZZ_X,u_!\, u^*\, j_!\, j^* A_X)\ar{r}\ar{d}&\RHom(\ZZ_X,j_!\, j^*A_X)\ar{d}\\
\RHom(\ZZ_X,u_!\, u^*A_X)\ar{r}&\RHom(\ZZ_X,v_!\, v^*A_X)
\end{tikzcd}
\]
where $u\colon X - H\hookrightarrow X$ is the complement of $H$
and $v\colon X - (H\cap Z)\hookrightarrow X$ is the complement of $H\cap Z$
(observe that $u_!\, u^*A_X\cong u_!\, u^*\, v_!\, v^*A_X)$
and $j_!\, j^*A_X\cong j_!\, j^*\, v_!\, v^*A_X)$).
\end{paragr}
\begin{defn}\label{def:Lefschetz filtration}
Let $X$ be an affine scheme of finite type over a field $k$ of dimension $d$.
A \emph{Lefschetz subscheme of $X$} is a closed subscheme $Z$ of $X$ of the form
$Z=W\cup H$, where $w\colon W\hookrightarrow X$ is a closed subscheme
of dimension $<d$, with affine complement,
that contains all irreducible components of $X$ of dimension $<d$,
such that $X-W$ is regular,
and $H$ is an hyperplane section that is transverse to the
motivic sheaf $j_!\, j^*\ZZ_X$
in the sense of \ref{def:good position},
where $j\colon X-W\to X$ denotes the complement of $W$ in $X$.
The \emph{rank of the pair $(X,Z)$} is the trace of the identity of $M(X,Z)[-d]$
in $\DM(\spec k)$; we will denote it
\[
\mathit{rk}(X,Z)=(-1)^d\mathit{Tr}\big(1_{M(X,Z)}\big)
\in\ZZ
\]
(this is always an integer because it coincides with the Euler characteristic
of its $\ell$-adic realizations \cite[Cor.~2.2.7]{MotGenBC}).
A \emph{Lefschetz pair} is a pair $(X,Z)$ that consists of an affine
scheme $X$ and of a Lefschetz subscheme $Z\subset X$.

A \emph{Lefchetz filtration} of $X$ is a finite sequence of closed subschemes
of the form
\[
\varnothing=X_{-1}\subset X_0\subset\ldots\subset X_{d-1}\subset X_d=X
\]
where, for $0\leq i\leq d$, each $X_{i-1}$ is a Lefschetz subscheme of $X_i$.

Lefschetz filtrations exist in practice thanks to the following proposition.
\end{defn}
\begin{prop}\label{prop:existence Lefschetz pairs}
Let $k$ be an infinite field.
Then, for any affine $k$-scheme $X$ and any closed subscheme $Y$ in $X$
of dimension $<\dim X$, there is a Lefschetz subscheme $Z$ of $X$
containing $Y$. 
If $k$ is a finite field, this is true only after pulling back $X$
along a finite extension of $k$.
\end{prop}
\begin{proof}
Define $W$ to be the union of $Y$ with the singular locus of $X$
as well as with the irreducible components of $X$ of dimension $<d$.
We may enlarge $W$ so that its complement if affine, and
then apply Corollary \ref{cor:abstract weak Lefschetz} to $F=j_!\, \ZZ_{X-W}$,
where $j\colon X-W\to X$ denotes the complement of $W$ in $X$,
in order
to find a section $H$ that is transverse to $F$, and put $Z=W\cup H$.
The case of a finite field is proved similarly, applying
Theorem \ref{thm:weak Lefschetz} instead of its corollary:
the dense open space $U$ of hyperplane sections will only have
rational points after a finite extension of $k$.
\end{proof}
\begin{lemma}\label{lemma:Lefschetz pairs are concentrated}
Let $X$ be an affine scheme of dimension $d\geq 0$,
$Z$ a Lefschetz subscheme of $X$, and $j\colon X-Z\hookrightarrow X$
the corresponding open immersion. Then $\mathit{rk}(X,Z)$
is a non-negative integer. Furthermore,
for any Artin-Verdier six-functors formalism $D$ defined on $k$-schemes
of finite type, the induced realization functor
\[
r_D\colon\DM(\spec k)\to D(\spec k)
\]
yields
\[
H^q(X,j_!\unit_{X-Z})\cong
\begin{cases}
0&\text{if $q\neq d$,}\\
r_D(M(X,Z)[-d])^\wedge&\text{if $q=d$.}
\end{cases}
\]
(where the cohomology is defined using the ordinary $t$-structure
as in Remark \ref{rem:perverse Artin}).
The rank of $H^d(X,j_!\unit_{X-Z})$ (defined as the trace of its identity)
is equal to the class of $\mathit{rk}(X,Z)$
in $\pi_0\map_{D(\spec k)}(\unit,\unit)$.

In particular, if $k$ is separably closed,
for any prime $\ell$ distinct from the characteristic of $k$,
we can compute the $\ell$-adic cohomology of the pair $(X,Z)$ motivically:
\[
H^q_\et(X,j_!\ZZ_\ell)
\cong\begin{cases}
0&\text{if $q\neq d$,}\\
\ZZ_\ell^{\mathit{rk}(X,Z)}&\text{if $q=d$.}
\end{cases}
\]
\end{lemma}
\begin{proof}
Let $D$ be an Artin-Verdier six-functors formalism defined on $k$-schemes
of finite type. We then have colimit preserving functors
\[
r_D\colon \DM(X)\to D(X)
\]
that commute with the six operations.
With the notations of Remark \ref{rem:perverse Artin},
if $M(X,Z)^\wedge\cong a_*j_!j^*\ZZ_X$
denotes the dual of $M(X,Z)$ in $\DM(\spec k)$, we thus
have
\[
r_D(M(X,Z)^\wedge)\cong\Gamma(X,j_!\, j^*\unit_{X})\, .
\]
Therefore,
\[
r_D(M(X,Z)[-d])\cong \Gamma(X,j_!\, j^*\unit_{X}[d])^\wedge
\]
is concentrated in degree $0$ by virtue of Theorem \ref{thm:Beilinson}
applied to $M=\ZZ_W$ and $U=W$ (with the notations of
Definition \ref{def:Lefschetz filtration}).
Since $r_D$ is symmetric monoidal, it preserves the formation of traces
of endomorphisms of dualizable objects.
For any $D$ such that the constructible
heart $D^\heartsuit_c(\spec k)$ is tannakian over $\QQ$,
we get that $\mathit{rk}(X,Z)$ is the Euler
characteristic of the image of $a_*j_!j^*\ZZ_X[d]$ through $r_D$,
and thus a non-negative integer. For instance, if $\ell$ is a prime
not equal to the characteristic of $k$, and if $\bar k$ is a
separable closure of $k$, we can associate to each $k$-scheme $X$ the
$\bar k$-scheme $\bar X$ obtained by pulling back $X$ to $\spec{\bar k}$.
The assignment
\[
X\mapsto D(\bar X,\QQ_\ell)
\]
then defines an Artin-Verdier six functors formalism $D$
such that $D^\heartsuit_c(\spec k)$ is the category of
finite dimensional $\QQ_\ell$-vector
spaces. This proves that $\mathit{rk}(X,Z)$ is a non-negative
integer.

Let us assume that $k=\bar k$ is separably closed
(in order to simplify the notations). We can
analyze the situation a little bit further as follows.
The morphisms of rings $\QQ_\ell\leftarrow\ZZ_\ell\to\FF_\ell$
induce isomorphisms on their Grothendieck rings
\[
\ZZ=K_0(\QQ_\ell)\cong K_0(\ZZ_\ell)\cong K_0(\FF_\ell)=\ZZ\, . 
\]
Therefore, for any perfect complex of $\ZZ_\ell$-modules $C$, we always have:
\[
\sum_q(-1)^q\dim_{\QQ_\ell} H^q(C)\otimes\QQ
=\sum_q(-1)^q\dim_{\FF_\ell}H^q(C\otimes^L_{\ZZ_\ell}\FF_\ell)\, .
\]
If ever both $C\otimes\QQ$ and $C\otimes^L_{\ZZ_\ell}\FF_\ell$ are concentrated in
cohomological degree $d$, this means that the dimensions of
$H^d(C\otimes^L_{\ZZ_\ell}\FF_\ell)$ and of
$H^d(C)\otimes\QQ$ are the same. 
A direct consequence can be seen in the exact sequence
\[
0\to H^q(C)\otimes_{\ZZ_\ell}\FF_\ell
\to H^q(C\otimes^L_{\ZZ_\ell}\FF_\ell)\to
\mathrm{Tor}_1(H^{q+1}(C),\FF_\ell)\to 0\, .
\]
This yields an isomorphism (an injective linear map from a space that has at least
the dimension of the codomain)
\[
H^q(C)\otimes_{\ZZ_\ell}\FF_\ell
\cong H^q(C\otimes^L_{\ZZ_\ell}\FF_\ell)
\]
for all $q$ (for $q\neq d$ this is even easier)
and tells us that $H^{q+1}(C)$ is free for all $q$.
Finally, we must have $H^q(C)\cong 0$ for any $q\neq d$
and $H^d(C)\cong\ZZ^r_\ell$ with $r$ the dimension of
$H^d(C)\otimes\QQ$ as a $\QQ_\ell$-vector space.

On constructible objects, the $\ell$-adic realization functor
\[
r_\ell\colon\DM_c(Y)\to D^b_c(Y,\ZZ_\ell)
\]
induces the realization functors
\[
\DM_c(Y)\to D^b_c(Y,\FF_\ell)
\quad \text{and}\quad \DM_c(Y)\to D^b_c(Y,\ZZ_\ell)\otimes\QQ=D^b_c(Y,\QQ_\ell)
\]
by the formul\ae
\[
M\mapsto r_\ell(M)\otimes^L_{\ZZ_\ell}\FF_\ell
\quad\text{and}\quad
M\mapsto r_\ell(M)\otimes\QQ
\]
respectively. Therefore, by virtue of Theorem \ref{thm:Beilinson}
for $D^b_c(-,\QQ_\ell)$
and $D^b_c(-,\FF_\ell)$, we see that we can
apply the discussion above with $C$ be the $\ell$-adic
realization of the motive $a_*j_!j^*\ZZ_X[d]$ in
the stable $\infty$-category
$D^b_c(\spec k,\ZZ_\ell)\cong\mathit{Perf}(\ZZ_\ell)$.
This proves the last part
of the lemma.
\end{proof}
\begin{rem}\label{rem:estimation}
Let $(X,Z)$ be a Lefschetz pair, with $X$ of dimension $d$
over a separably closed field.
Let $j\colon X-Z\to X$ be the corresponding embedding.
Given any prime number $\ell$ distinct from the
characteristic of the ground field, we have
a long exact sequence
\[
\cdots\to H^q_\et(X,j_!\ZZ_\ell)\to H^q_\et(X,\ZZ_\ell)
\to H^q_\et(Z,\ZZ_\ell)\to H^{q+1}_\et(X,j_!\ZZ_\ell)\to \cdots
\]
By virtue of the preceding lemma, this yields isomorphisms
\[
H^q_\et(X,\ZZ_\ell)\cong H^q_\et(Z,\ZZ_\ell)\quad\text{for $q\neq d, d-1$,}
\]
and an exact sequence of the form below.
\[
0\to H^{d-1}_\et(X,\ZZ_\ell)
\to H^{d-1}_\et(Z,\ZZ_\ell)\to \ZZ_\ell^{\mathit{rk}(X,Z)}\to
H^{d}_\et(X,\ZZ_\ell)\to 0
\]
The map $H^{d-1}_\et(X,\ZZ_\ell)
\to H^{d-1}_\et(Z,\ZZ_\ell)$ is thus a split monomorphism inducing an isomorphism
on torsion subgroups
$H^{d-1}_\et(X,\ZZ_\ell)_{\mathit{tors}}
\cong H^{d-1}_\et(Z,\ZZ_\ell)_{\mathit{tors}}$.
\end{rem}
\begin{thm}\label{thm:Betti estimation}
Let $X$ be an affine scheme of Krull dimension $d$,
of finite type over a separably closed
field $k$, and
equipped with a Lefschetz filtration
\[
\varnothing=X_{-1}\subset X_0\subset\ldots\subset X_{d-1}\subset X_d=X\, .
\]
For any prime $\ell$ distinct from the characteristic of $k$, and
$0\leq q\leq d$, we have:
\[
\mathit{rk}\, H^q(X,\ZZ_\ell)
\leq \mathit{rk}\, H^q(X_q,\ZZ_\ell)
\leq\mathit{rk}(X_q,X_{q-1})\, .
\]
\end{thm}
\begin{proof}
This follows right away
from Lemma \ref{lemma:Lefschetz pairs are concentrated}
and from Remark \ref{rem:estimation} by induction on $d$.
\end{proof}
\begin{rem}
Together with Proposition \ref{prop:existence Lefschetz pairs},
this gives uniform bounds of Betti numbers
for all quasi-projective schemes over a field,
using Jouanolou's trick \cite[Lemme 1.5]{Jou}: the latter proves
that any quasi-projective scheme is
$\mathbf{A}^1$-homotopy equivalent to an affine one.

Lefschetz pairs may be used
to prove the existence of uniform bounds of Betti numbers of
any constructible motive; see Corollary \ref{cor:bound Betti numbers of motives}
below.
\end{rem}
\begin{paragr}
We can interpret the existence of Lefschetz filtrations as an approximation
of the motivic $t$-structure on $\DM(X)$ (whose existence would imply
Grothendieck's standard conjectures \cite{Standard}). In order to simplify our task,
we will work with $\QQ$-linear coefficients (it is straightforward to produce an
integral version from any $\QQ$-linear one using the rigidity theorem).

We observe that objects of the form $M(X)(n)$, with $X$ affine of finite type over
the base field $k$ and $n\in\ZZ$ are compact and generate $\DM(\spec k)$ as
a compactly generated stable $\infty$-category.
If $k$ is infinite, this implies that the objects of the form $M(X,Z)(n)$
with $(X,Z)$ a Lefschetz pair and $n\in\ZZ$ also form a generating
family. In order to allow any field, we proceed as follows.
\end{paragr}
\begin{prop}\label{prop:Lefschetz enough}
Let $k$ be any field.
We define the family $\mathcal{G}_k$ as follows:
this is the family of motives of the form $f_*M(X,Z)[-\dim X](n)$,
where $f\colon \spec K\to\spec k$ is a Galois extension of the base field,
with $(X,Z)$ a Lefschetz pair defined over $K$, and $n\in\ZZ$.
Then $\mathcal{G}_k$ is a family of compact generators of $\DM(k,\QQ)$
as a cocomplete stable $\infty$-category.
\end{prop}
\begin{proof}
The case where $k$ is infinite is obvious from
Proposition \ref{prop:existence Lefschetz pairs}:
any affine scheme has a Lefschetz filtration
and given a Lefschetz filtration by subschemes $X_i$, one can
reconstruct the motive of $X$ from the $M(X_i,X_{i-1})$'s.
Since the motives of the form $M(X)(n)$ with $X$ affine form a generating
family, this proves the claim if $k$ is infinite,
in particular for $k$ separably closed.
In general,
we know that, if $\bar k$ is a separable closure of $k$, then the pullback
functor
\[
\DM(\spec k,\QQ)\to\DM(\spec{\bar k},\QQ)
\]
is conservative and preserve the formation of internal Hom's of the form $\uHom(A,B)$
as long as $A$ is compact; see \cite[Lemma 4.3.7]{CD3}. By a continuity argument
(as in \cite[Prop.~4.3.1]{CD3}), it is now easy to see that
it suffices to prove that $\mathcal{G}_k$ generates $\DM(\spec{k},\QQ)$
if $\mathcal{G}_{\bar k}$ generates $\DM(\spec{\bar k},\QQ)$,
which we already know.
\end{proof}
\begin{cor}\label{cor:bound Betti numbers of motives}
Let $M$ be any constructible object in $\DM(\spec k,\QQ)$.
There exists a constant $C>0$ and integer $a\leq b$ with the following
properties.
For any
$\QQ$-linear Artin-Verdier six-functors formalism $D$ defined over $k$-schemes
of finite type with tannakian heart $D^\heartsuit_c(\spec k)$,
the induced functor
\[
r_D\colon \DM(\spec k,\QQ)\to D(\spec k)
\]
sends $M$ to an object $r_D(M)$ of cohomological amplitude $[a,b]$
with respect to the ordinary $t$-structure and
so that the rank of each cohomology objects $H^q(r_D(M))$
(i.e. the traces of their identity) is bounded by $C$.
\end{cor}
\begin{proof}
The preceding proposition means that
such an $M$ is obtained by finitely many operations
(such as extracting a direct factor in a finite coproducts, taking cofibers, shifting
by some integer) from finitely many objects of the form
$f_*M(X,Z)[-\dim X](n)$,
where $f\colon \spec K\to\spec k$ is a Galois extension of the base field,
with $(X,Z)$ a Lefschetz pair defined over $K$, and $n\in\ZZ$.
Therefore, it suffices
to prove the result for $M=f_*M(X,Z)[-\dim X](n)$.
Since $f_*$ is an exact functor for $f$ finite, we conclude
with Lemma \ref{lemma:Lefschetz pairs are concentrated}.
\end{proof}
\begin{paragr}
We call \emph{standard generators} objects of the form
\[
A^\wedge_1\otimes\cdots\otimes A^\wedge_m\otimes B_1\otimes\cdots\otimes B_n
\]
where each $A_i$ and $B_j$ are in $\mathcal{G}_k$ and $(-)^\wedge$ is the duality operator on dualizable objects in $\DM(\spec k,\QQ)$.
We define $\DM^{\leq 0}(\spec k,\QQ)$ as the smallest full subcategory
of $\DM(\spec k,\QQ)$ stable under small colimits and stable under
extensions spanned by standard generators. This is a compactly
generated presentable $\infty$-category. Therefore, the inclusion functor
\[
\DM^{\leq 0}(\spec k,\QQ)\hookrightarrow\DM(\spec k,\QQ)
\]
has a right adjoint
\[
\tau^{\leq 0}\colon\DM(\spec k,\QQ)\to\DM^{\leq 0}(\spec k,\QQ)
\]
We define $\DM^{\geq 0}(\spec k,\QQ)$ as the full subcategory of objects $B$ such that
\[
\map(A[1],B)=0
\]
for any $A$ in $\DM^{\leq 0}(\spec k,\QQ)$. This defines a $t$-structure on
$\DM(\spec k,\QQ)$. By construction
(and thanks to Lemma \ref{lemma:Lefschetz pairs are concentrated}), for any
$\QQ$-linear Artin-Verdier six-functors formalism $D$ defined over $k$-schemes
of finite type, the induced functor
\[
r_D\colon \DM(\spec k,\QQ)\to D(\spec k)
\]
is right $t$-exact and induces a symmetric monoidal colimit preserving functor
\[
r_D\colon \DM^{\leq 0}(\spec k,\QQ)\to D^{\leq 0}(\spec k)\, .
\]
This $t$-structure on $\DM(\spec k,\QQ)$ is an approximation
of the conjectural motivic $t$-structure - if we are optimistic, we could
hope it is the right one. We would like to do something similar
in $\DM(Y,\QQ)$ for any $Y$ of finite type over a field:
approximating the perverse $t$-structure. For this, we need
to define a relative notion of Lefschetz pairs.
\end{paragr}
\begin{paragr}
Let $Y$ be a regular scheme of finite type over $k$.
An affine $Y$-scheme of finite type $X$ is \emph{admissible} if
the structural map $a\colon X\to Y$ is flat and
if the motivic sheaf $a_*(\QQ_X)$ is dualizable in $\DM(Y,\QQ)$
(up to replacing $Y$ by a dense open subscheme, any flat
scheme over $Y$ is admissible; by Grothendieck's generic flatness
theorem, this means that, up to replacing $Y$ by a dense open subscheme,
any $Y$-scheme of finite type can be made admissible).

Given an admissible $Y$-scheme $X$
with structural morphism $a\colon X\to Y$, we define
\[
M_Y(X)=a_*(\ZZ_X)^\wedge \, .
\]
Given a closed subscheme $Z$ of $X$ that is admissble over $Y$ as well,
we define $M_Y(X,Z)$ by a cofiber sequence
\[
M_Y(Z)\to M_Y(X)\to M_Y(X,Z)
\]
\end{paragr}
\begin{defn}
Let $Y$ be a regular scheme of finite type over $k$.
Let $a\colon X\to Y$ be flat morphism between $k$-schemes of finite type
of relative dimension $d$.
A \emph{Lefschetz subscheme of $X$ over $Y$} is a closed subscheme $Z$ of $X$
that is flat over $Y$, of the form $W\cup H$, where $W$ is a closed
subscheme of $X$, with affine complement, and
that if flat over $Y$, with $\dim(W)=\dim(X)-1$,
which contains all irreducible components of $X$ of dimension $\dim(X)$,
such that $X - W$ is regular,
and $H$ is an hyperplane section that is transverse to the
motivic sheaf $j_!\, j^*\ZZ_X$
in the sense of \ref{def:good position}, where $j\colon X- W\to X$ is the
embedding of the complement of $W$. We will then say that $(X,Z)$
is a Lefschetz pair over $Y$. A Lefschetz pair $(X,Z)$ is
\emph{admissible} over $Y$ if both $X$ and $Z$ are admissible.
\end{defn}
\begin{paragr}
Let $a\colon X\to Y$ be the structural morphism of the admissible affine
$Y$-scheme $X$.
We choose an affine open embedding $e\colon X\to\bar X$ followed by a projective
morphism $\bar a\colon \bar X\to Y$ with $a=\bar a \, e$.
Possibly after shrinking $Y$, we may assume that $\bar X$ is flat over $Y$.
Let us consider an admissible Lefschetz subscheme $Z$ in $X$, and
write $p\colon \bar X -\bar H\to Y$ for the structural map of the complement
of the chosen hyperplane section $H$.
We have
\[
a_*\, j_!\, \QQ_{X-Z}\cong p_!\, e_*\, j_!\QQ_{X-W}\, .
\]
We will say that $(X,Z)$ is \emph{tight} if the conclusion of
Proposition \ref{prop:basic vanishing} holds without shrinking $Y$
any further
for the realizations of $a_*\, j_!\, \QQ_{X-Z}$ and of $p_!\, e_*\, j_!\QQ_{X-W}$
through any $\QQ$-linear Artin-Verdier six-functors formalism $D$.
The aforementioned proposition says that, given any
admissible Lefschetz pair $(X,Z)$ over $Y$,
there always exists a dense
open subscheme of $Y$ over which $(X,Z)$ becomes tight.
\end{paragr}
\begin{prop}\label{prop:tight nice}
If $(X,Z)$ is a tight admissible Lefschetz pair over a regular scheme $Y$,
with $X$ of relative dimension $d$ over $Y$, 
then, for any $\QQ$-linear Artin-Verdier six-functors formalism $D$,
the induced functor
\[
r_D\colon \DM(Y,\QQ)\to D(Y)
\]
sends
\[
M_Y(X,Z)^\wedge[d]\cong a_*\, j_!\, \QQ_{X-Z}[d]
\]
to a dualizable object that is concentrated in degree $0$
for the ordinary $t$-structure in $D(Y)$.
\end{prop}
\begin{proof}
This is a direct consequence of the definition,
exactly as in the proof of Theorem \ref{thm:Beilinson} and of
Lemma \ref{lemma:Lefschetz pairs are concentrated}.
\end{proof}
\begin{paragr}
Let $Y$ be a scheme of finite type over $k$
with $Y_\mathit{red}$ regular, and let $\imath\colon Y_\mathit{red}\to Y$
be the canonical embedding.
We denote by $\mathcal{G}_Y$ the family of objects of the form
\[
\imath_*\, f_* M_{Y'}(X',Z')(n)[-d]
\]
with $f\colon Y'\to Y_{\mathit{red}}$ a finite flat morphism,
$Y'$ regular, $(X',Z')$ a tight admissible Lefschetz pair over $Y'$,
$d$ the relative dimension of $X$ over $Y'$
(seen as a locally constant function), and $n$ any integer.
As we did previously we define \emph{standard generators} over $Y$
objects of the form
\[
A^\wedge_1\otimes\cdots\otimes A^\wedge_m\otimes B_1\otimes\cdots\otimes B_n
\]
where each $A_i$ and $B_j$ are in $\mathcal{G}_Y$ and $(-)^\wedge$ is the duality operator on dualizable objects in $\DM(Y,\QQ)$.

For a scheme of finite type $X$ over $k$, we define the
\emph{perverse generating family} $P(X)$ by noetherian induction on dimension
as follows.
For $\dim X\leq 0$, $P(X)=\mathcal{G}_X$.
For $\dim X>0$, $P(X)$ is the collections of objects of the following two
possible forms:
\begin{itemize}
\item $i_*(M)$ for $i\colon Z\to X$ a nowhere dense closed immersion
and $M$ in $P(Z)$;
\item $j_!M[\dim Y]$ for $j\colon Y\to X$ an affine dense open immersion with
$Y_\mathit{red}$ affine and regular, and $M$ is a standard generator over $Y$.
\end{itemize}
\end{paragr}
\begin{prop}\label{prop:existence tight Lefschetz pairs}
The family $P(X)$ forms a family of compact generators in $\DM(X,\QQ)$.
\end{prop}
\begin{proof}
We already know this is true in the case where $X$ is the spectrum of a field.
The case where $X$ is of dimension $\leq 0$ follows right away.
Theorem \ref{thm:weak Lefschetz} together with Proposition
\ref{prop:basic vanishing} and Grothendieck's
generic flatness theorem ensure that tight admissible Lefschetz pairs
exist in abundance generically.
The proposition follows easily from \cite[Prop.~4.3.17]{CD3} applied
to $\DM$ by continuity arguments.
\end{proof}
\begin{paragr}\label{paragr:t-structures}
We define $^p\DM^{\leq 0}(X,\QQ)$ as the smallest full subcategory
of $\DM(X,\QQ)$ stable under small colimits and stable under
extensions spanned by $P(X)$. This is a compactly
generated presentable $\infty$-category. Therefore, the inclusion functor
\[
^p\DM^{\leq 0}(X,\QQ)\hookrightarrow\DM(\spec k,\QQ)
\]
has a right adjoint
\[
\tau^{\leq 0}\colon\DM(X,\QQ)\to{^p\DM^{\leq 0}(X,\QQ)}
\]
We define $^p\DM^{\geq 0}(X,\QQ)$ by orthogonality
as the full subcategory of objects $B$ such that
\[
\map(A[1],B)=0
\]
for any $A$ in $ ^p\DM^{\leq 0}(X,\QQ)$. This defines a
$t$-structure on $\DM(X,\QQ)$ that would be a candidate for
the perverse motivic $t$-structure.

If $D$ is any
$\QQ$-linear Artin-Verdier six-functors formalism defined over $k$-schemes
of finite type, the perverse $t$-structure on $D_c$
extends to a perverse $t$-structure on $D$ such that the embedding
\[
D_c(X)\hookrightarrow D(X)
\]
is $t$-exact for all $X$: simply define $^pD^{\leq 0}(X)$
as the smallest full subcategory
of $D(X)$ stable under small colimits and stable under
extensions spanned by perverse sheaves in $D_c(X)$.
For any $k$-scheme of finite type $X$, Proposition \ref{prop:tight nice}
imply that the induced functor
\[
r_D\colon \DM(X,\QQ)\to D(X)
\]
is right $t$-exact and induces a colimit preserving functor
\[
r_D\colon {^p\DM^{\leq 0}(X,\QQ)}\to {^pD^{\leq 0}(X)}\, .
\]
There is an analogous construction for a candidate of the
ordinary $t$-structure on $\DM(X,\QQ)$. Define $G(X)$ as follows:
$G(X)=P(X)$ for $X$ of dimension $\leq 0$, and for $X$ of dimension $d>0$,
define $G(X)$ as the collection of objecrs of the form
\begin{itemize}
\item $i_*(M)$ for $i\colon Z\to X$ a nowhere dense closed immersion
and $M$ in $G(Z)$;
\item $j_!M$ for $j\colon Y\to X$ an affine dense open immersion with
$Y_\mathit{red}$ affine and regular, and $M$ is a standard generator over $Y$.
\end{itemize}
We define $\DM^{\leq 0}(X,\QQ)$ as the smallest full subcategory
of $\DM(X,\QQ)$ stable under small colimits and stable under
extensions spanned by $P(X)$. This defines as above a $t$-structure
such that $r_D$ is right $t$-exact for any $D$
(with respect to the ordinary $t$-structure on $D$).
By construction, $\QQ_X$ belongs to $\DM^{\leq 0}(X,\QQ)$ and
the tensor product defines a functor that preserves
colimit in each variables
\[
(-)\otimes(-)\colon\DM^{\leq 0}(X,\QQ)^\mathit{op}
\times \DM^{\leq 0}(X,\QQ)
\to\DM^{\leq 0}(X,\QQ)\, .
\]
A wild optimistic guess would be that any element of $G(X)$ is in the heart
of this $t$-structure.
\end{paragr}
\bibliographystyle{amsalpha}
\bibliography{uniform}
\end{document}